\documentclass[11 point, a4paper, reqno]{amsart}
\usepackage[english]{babel}
\usepackage{amsthm, fullpage}
\usepackage{graphicx}
\usepackage{amsfonts}
\usepackage{eufrak}
\usepackage{amsopn}
\usepackage{amsmath, amsfonts, amssymb, amscd, amsthm}
\usepackage{enumitem}
\usepackage{setspace}
\usepackage{mathtools}

\usepackage[all]{xy}

\newtheorem{theorem}{Theorem}[section]
\newtheorem{lemma}[theorem]{Lemma}
\newtheorem{proposition}[theorem]{Proposition}
\newtheorem{corollary}[theorem]{Corollary}
\theoremstyle{definition}

\theoremstyle{remark}
\newtheorem{remark}[theorem]{Remark}
\theoremstyle{example}

\theoremstyle{note}

\numberwithin{equation}{section}

\DeclareMathOperator{\Hom}{Hom}

\DeclareMathOperator{\GL}{GL}

\DeclareMathOperator{\M}{M}

\DeclareMathOperator{\Aut}{Aut}
\DeclareMathOperator{\tr}{Tr}
\DeclareMathOperator{\reg}{reg}
\DeclareMathOperator{\coh}{H}

\usepackage[colorlinks]{hyperref}

\usepackage{comment}

\begin{document}

\title{Dualizing involutions on the $n$-fold metaplectic cover of $\GL(2)$}
\author{Kumar Balasubramanian} %\thanks{* indicates corresponding author}
\thanks{Research of Kumar Balasubramanian is supported by the SERB grant: CRG/2023/000281.}
\address{Kumar Balasubramanian\\
Department of Mathematics\\
IISER Bhopal\\
Bhopal, Madhya Pradesh 462066, India}
\email{bkumar@iiserb.ac.in}

\author{Renu Joshi} 
\address{Renu Joshi\\
Department of Mathematics\\
IISER Bhopal\\
Bhopal, Madhya Pradesh 462066, India}
\email{renu16@iiserb.ac.in}

\author{Sanjeev Kumar Pandey} 
\address{Sanjeev Kumar Pandey\\
Department of Mathematics\\
IISER Bhopal\\
Bhopal, Madhya Pradesh 462066, India}
\email{sanjeevpandey@iiserb.ac.in}

\author{Varsha Vasudevan} 
\address{Varsha Vasudevan\\
Department of Mathematics\\
IISER Bhopal\\
Bhopal, Madhya Pradesh 462066, India}
\email{varsha23@iiserb.ac.in}


\maketitle
\begin{abstract}
Let $F$ be a non-Archimedean local field of characteristic zero and  $G=\GL(2,F)$. Let $n\geq 2$ be a positive integer and $\widetilde{G}=\widetilde{\GL}(2,F)$ be the $n$-fold metaplectic cover of $G$. Let $\pi$ be an irreducible smooth representation of $G$ and $\pi^{\vee}$ be the contragredient of $\pi$. Let $\tau$ be an involutive anti-automorphism of $G$ satisfying $\pi^{\tau}\simeq \pi^{\vee}$. In this case, we say that $\tau$ is a dualizing involution. A well known theorem of Gelfand and Kazhdan says that the standard involution $\tau$ on $G$ is a dualizing involution. In this paper, we show that any lift of the standard involution to $\widetilde{G}$ is a dualizing involution if and only if $n=2$.
\end{abstract}

\section{Introduction}

Let $F$ be a non-Archimedean local field of characteristic $0$ and $G=\GL(n,F)$. For $g\in G$, we let $g^{\top}$ denote the transpose of the matrix $g$, and $w_{0}$ to be the matrix with anti-diagonal entries equal to one. Let $\tau: G\rightarrow G$ be the map $\tau(g)=w_{0}g^{\top}w_{0}$. Clearly $\tau$ is an anti-automorphism of $G$ such that $\tau^{2}=1$. We call $\tau$ the \textit{standard involution} on $G$. Let $(\pi,V)$ be an irreducible smooth complex representation of $G$. We write $(\pi^{\vee}, V^{\vee})$ for the smooth dual or the contragredient of $(\pi,V)$. For $\beta$ an anti-automorphism of $G$ such that $\beta^{2}=1$, we let $\pi^{\beta}$ to be the twisted representation defined by \[\pi^{\beta}(g)=\pi(\beta(g^{-1})).\]

The following is an old result of Gelfand and Kazhdan.

\begin{theorem}[Gelfand-Kazhdan] Let $\tau$ be the standard involution on $G$. Then \[\pi^{\tau}\simeq \pi^{\vee}.\]
\end{theorem}

We refer the reader to Theorem 2 in \cite{GelKaj} for a proof of the above result. \\

If $\beta$ is any anti-automorphism of $G$ such that $\beta^{2}=1$, and satisfies $\pi^{\beta}\simeq \pi^{\vee}$, then we call $\beta$ a \textit{dualizing involution}. The above result of Gelfand and Kazhdan says that the standard involution $\tau$ on $G$ is a dualizing involution. \\

Let $r \geq 2$ be a positive integer. For $ n\geq 2$, let $\widetilde{G}$ be the $r$-fold metaplectic cover of $G$ (see Chapter 0 in \cite{KazPat[2]} for the general definition). It is well known that (see Proposition 3.1 in \cite{KazPat}) the standard involution $\tau$ on $G$ has at least one lift to the metaplectic group. A natural and an interesting question that arises is whether the lifts of the standard involution are themselves dualizing involutions. In the case when $\widetilde{G}$ is the $2$-fold metaplectic cover of $G = \GL(2, F)$, it is known that any lift of $\tau$ to $\widetilde{G}$ is dualizing (see \cite{KumAji}, \cite{KumEkt}).\\

In this paper, we address this question when $\widetilde{G}$ is the $n$-fold metaplectic cover of $G = \GL(2, F)$ (explained later) for $n\geq 2$. It follows from Proposition 1 in \cite{Kab} that there are many lifts of the standard involution to $\widetilde{G}$. In particular, for each $\alpha\in F^{\times}$, we have a lift $\sigma_{\alpha}$ (see Section~\ref{main theorem} for the definition) of $\tau$ to $\widetilde{G}$. We show that the involutions $\sigma_{\alpha}$ are dualizing if and only if $n=2$. To be precise, we prove \\

\begin{theorem}[Main Theorem]\label{main-theorem} Let $\pi$ be any irreducible admissible genuine representation of $\widetilde{G}$. For $\alpha\in F^{\times}$, let $\sigma_{\alpha}$ be the lift of $\tau$ to $\widetilde{G}$. Then
\[\pi^{\sigma_{\alpha}}\simeq \pi^{\vee} \]
if and only if $n=2$. 
\end{theorem}

The above result raises the question, if a similar result holds for the $r$-fold covering $\widetilde{G}$ of $G = \GL(n)$ for $r,n\geq 2$. To be precise, a lift of the standard involution $\tau$ of $G$ to $\widetilde{G}$ is a dualizing involution if and only if $r=2$. We hope to address some of these questions in future.\\

The paper is organized as follows. In Section~\ref{preliminaries}, we recall the basic properties of the Hilbert symbol and use it to establish some more properties that we need. We give the definition of the metaplectic group that we work with and discuss its topology. We recall a result of Harish-Chandra about the character of an admissible representation and give references for analogous results in the setting of metaplectic groups. We prove some basic properties of the character that we need. We also mention a few results of Kable about the lifts of an automorphism. In Section~\ref{sigma and its properties}, we explicitly define lifts of the standard involution and discuss some properties. In Section~\ref{main theorem}, we prove the main result of this paper.

\section{preliminaries}\label{preliminaries}

In this section, we set up some preliminaries which we need and recall a few results which will be used throughout this paper.

\subsection{$n$-th order Hilbert Symbol and its properties}

Let $F$ be a non-Archimedean local field of characterisic zero and $F^{\times}$ be the group of non-zero elements in $F$. For $n\in \mathbb{N}$, let $\mu_{n}=\{x\in F^{\times} \mid x^n=1\}$ be the subset of $n$-th roots of unity in $F^{\times}$. Suppose also that $|\mu_{n}|=n$. We let \[\langle \, , \, \rangle: F^{\times}\times F^{\times}\rightarrow \mu_{n}\] denote the $n$-th order Hilbert symbol. In this section, we recall some fundamental properties of the $n$-th order Hilbert symbol and use it to derive some simple consequences that we need. \\

The following basic properties (see \cite{KazPat[2]}) of the Hilbert symbol are well known. We record them in the proposition below.

\begin{proposition}\label{fundamental properties of the Hilbert symbol} Let $a,b,c\in F^{\times}$. The Hilbert symbol satisfies
\begin{enumerate}
\item[\emph{1)}] $\langle ab,c\rangle = \langle a,c\rangle \langle b,c\rangle$.
\vspace{0.1 cm}
\item[\emph{2)}] $\langle a,b\rangle \langle b,a\rangle=1$. 
\vspace{0.1 cm}
\item[\emph{3)}] $\langle a,1-a\rangle=1$, if $1-a\in F^{\times}$.
\vspace{0.1 cm}
\item[\emph{4)}] $\{a \in F^{\times} \mid \langle a,x\rangle = 1$ for all $x\in F^{\times}\}=(F^{\times})^n$, where $(F^{\times})^n=\{a^n \mid a \in F^{\times}\}$. 
\end{enumerate}
\end{proposition}
 
The following properties of the Hilbert symbol can be derived easily from properties $1)$ to $4)$ mentioned in Proposition~\ref{fundamental properties of the Hilbert symbol} above. 

\begin{proposition}\label{properties of the Hilbert symbol} For $a,b,c \in F^{\times}$, the Hilbert symbol satisfies
\begin{enumerate}[label=\emph{(\roman*)}]
\item $\langle a,bc\rangle = \langle a,b\rangle \langle a,c\rangle$.
\vspace{0.1 cm}
\item $\langle a^m,b\rangle =\langle a,b\rangle^m=\langle a,b^m\rangle$ for any $m\in \mathbb{N}$.
\vspace{0.1 cm}
\item $\langle 1-a,a\rangle=1$ if $a\neq 1$, and $\langle a, 1 \rangle=\langle 1, a\rangle=1$.
\vspace{0.1 cm}
\item $\langle b,a\rangle=\langle a,b\rangle^{-1}=\langle a^{-1},b\rangle=\langle a,b^{-1}\rangle$.
\vspace{0.1 cm}
\item $\langle -a,a\rangle=\langle a,-a\rangle=\langle a,a^2\rangle=1$.
\vspace{0.1 cm}
\item If $\langle a, x \rangle =1$ for all $x\in F^{\times},$ then $a\in (F^{\times})^{n}$. In other words, if $a\notin (F^{\times})^{n}$, there exists $x_0 \in F^{\times}$ such that $\langle a, x_0 \rangle \neq 1$.
\end{enumerate}
\end{proposition}

\begin{proof} \hfill \\
\begin{enumerate}[label=(\roman*)]
\item Using properties 1) and 2) in Proposition~\ref{fundamental properties of the Hilbert symbol}, we have \begin{align*}
    \langle a,bc\rangle &= \langle bc, a \rangle^{-1} \\
      &=\langle b,a \rangle^{-1} \langle c, a \rangle^{-1} \\
      &=\langle a,b\rangle \langle a,c\rangle. 
\end{align*}

\item It follows from (i) and property 1). \\

\item We know that $\langle 1-a,a\rangle \langle a, 1-a\rangle=1$, so the result follows from property 3). By (i), we have $$\langle a,1\rangle=\langle a,1\cdot 1\rangle=\langle a,1\rangle \langle a,1\rangle.$$ Hence $\langle a,1\rangle=1$. Consequently, $\langle 1,a\rangle=\langle a,1\rangle^{-1}=1$. \\

\item Since $\langle a,b\rangle \langle b,a\rangle=1$, we obtain $\langle b,a\rangle=\langle a,b\rangle^{-1}$. Using (iii) and property 1), we have $$\langle a, b \rangle \langle a^{-1}, b \rangle=\langle aa^{-1}, b\rangle =\langle 1, a \rangle =1.$$ Similarly $$\langle a, b \rangle \langle a, b^{-1} \rangle=\langle a, bb^{-1}\rangle =\langle a, 1 \rangle =1.$$
\vspace{0.2 cm}

\item If $a=1$, we are done using (iii). For $a\neq 1$, observe that $-a=\frac{1-a}{1-a^{-1}}$. Then 
\begin{align*}
   \langle a,-a \rangle &= \left< a, \frac{1-a}{1-a^{-1}}\right >\\
   &=\langle a,1-a \rangle \langle a,(1-a^{-1})^{-1} \rangle & \text{(using (i))}\\
   &=\langle a,1-a \rangle \langle a^{-1},1-a^{-1} \rangle & \text{(using (iv))}\\
   &= 1& \text{(using property 3))}. 
\end{align*}
Thus, we also have $\langle a,a^2\rangle=\langle a, -a\rangle \langle a, -a\rangle =1$. \\

\item Follows from property 4).
\end{enumerate}
\end{proof}

\subsection{The Metaplectic Group} Let $F$ be a non-Archimedean local field of characteristic zero and $G=\GL(2,F)$. For $n\geq 2$, let $\mu_{n}\subset F^{\times}$ be the subset of $n$-th roots of unity in $F$. Suppose also that $|\mu_{n}|=n$. Throughout, we write $\widetilde{G}=\widetilde{\GL}(2,F)$ for the $n$-fold metaplectic cover of $G$. In this section, we recall the definition of $\widetilde{G}$ and a few basic facts about its topology. \\

Throughout, we write $\mathfrak{o}$ for the ring of integers in $F$, $\mathfrak{p}$ for the unique maximal ideal in $\mathfrak{o}$, $\varpi$ for the generator of $\mathfrak{p}$ and $k_{F}$ for the finite residue field. We write $ord$ for the valuation on $F$.\\

The first explicit construction of the $n$-fold metaplectic cover of $G=\GL(2,F)$ was given by Kubota in \cite{Kub} by concretely describing a $2$-cocyle on $G$ taking values in $\mu_{n}$. For $g_{1}, g_{2}, m\in G$, the following simpler version of the Kubota cocycle $c: G\times G \rightarrow \mu_n$ defined as

$$
c(g_1, g_2)= \left<\frac{X(g_1g_2)}{X(g_1)},  \frac{X(g_1g_2)}{X(g_2)\Delta(g_1)}\right>,
$$
where
$$
X(m)= \begin{cases} m_{21}& \text{if}\; m_{21}\neq 0\\m_{22} &\text{otherwise}
\end{cases}
\quad \textrm{ and } \Delta(g_1)=\det(g_1)
$$

was given by Kazhdan and Patterson in \cite{KazPat}. Throughout this paper, we take $\widetilde{G}=\widetilde{\GL}(2,F)$ to be the central extension of $G$ by $\mu_{n}$ determined by the $2$-cocycle $c$. Since $G$, $\mu_{n}$ are locally compact groups, Mackey's theorem (see Theorem 2 in \cite{Mac}) implies that $\widetilde{G}$ is a locally compact topological group and defines a topological central extension of $G$ by $\mu_{n}$. We call the group $\widetilde{G}$ constructed above as ``the" metaplectic group. \\

It can be shown that the topology on $\widetilde{G}$ has a neighborhood base at the identity consisting of compact open subgroups (see Lemma 3 in \cite{Kab}). Before we give the construction of this basis, we recall a few preliminaries. \\

Let $p:\widetilde{G}\rightarrow G$ be the projection of $\widetilde{G}$ onto $G$. A map $\ell: G\rightarrow \widetilde{G}$ is called a \textit{section} if $p\circ \ell = 1_{G}$. Given a subgroup $H$ of $G$, we say that $\widetilde{G}$ \textit{splits} over $H$ if there exists a homomorphism $h: H\rightarrow \widetilde{G}$ such that $p\circ h =1_{H}$.\\

Let $\ell: G\rightarrow \widetilde{G}$ be the map $\ell(g)=(g,1)$. Then $\ell$ is a section and is called the natural or preferred section. For $g=\begin{bmatrix} a & b \\ c & d \end{bmatrix}\in G$, define $s: G \rightarrow \mu_{n}$ as

\begin{equation*}\label{definition of s(g)}
s(g) = \begin{cases} \langle c, d\Delta(g)\rangle, & \textrm{ if } cd\neq 0 \textrm{ and $ord(c)$ is odd}\\
\hfil 1, & \textrm{ otherwise. } \end{cases}
\end{equation*}

Let $K_{0}=\GL(2, \mathfrak{o})$ be the maximal compact subgroup in $G$, and for $\lambda\geq 1$, let $K_{\lambda}= 1+\varpi^{\lambda}\M(n,\mathfrak{o})$. It is known that $\{K_{\lambda}\}_{\lambda\geq 0}$ is a neighborhood base at the identity element in $G$ consisting of compact open subgroups. We can use this base to define a neighborhood base at the identity in $\widetilde{G}$. For a suitable choice of $\lambda_{0}$, it can be shown that $\widetilde{G}$ splits over $K_{\lambda}$, where $\lambda\geq \lambda_{0}$ (see Proposition 2.8 in \cite{Gel}, Theorem 2 in \cite{Kub}). Define $\kappa: K_{\lambda_{0}}\rightarrow \widetilde{G}$ as $\kappa(k)=(k,s(k))$. Let $K_{\lambda_{0}}^{*}=\kappa(K_{\lambda_{0}})$ and for $\lambda > \lambda_{0}$, $K^{*}_{\lambda}= K_{\lambda_{0}}^{*}\cap p^{-1}(K_{\lambda})$. It can be shown that $\{K^{*}_{\lambda}\}_{\lambda\geq \lambda_{0}}$ is a neighborhood base at the identity in $\widetilde{G}$. \\

\subsection{Character of an admissible representation}

Let $F$ be a non-Archimedean local field of characteristic $0$ and $G=G(F)$ be a connected reductive algebraic group defined over $F$. We let $(\pi,V)$ be an irreducible smooth complex representation of $G$. It can be shown that such representations are always admissible. For an admissible representation $(\pi,V)$, we can define a suitable notion of a character. Before we proceed further, we set up some notation and recall an important result of Harish-Chandra. \\

Throughout we let $G=G(F)$ and $(\pi,V)$ to be an irreducible smooth representation of $G$. We let $C^{\infty}_{c}(G)$ to be the space of all locally constant complex valued functions on $G$ with compact support. For $f\in C^{\infty}_{c}(G)$, we let $\pi(f): V\rightarrow V$ denote the linear operator given by
\[\pi(f)v= \int_{G}f(g)\pi(g)vdg, \quad v\in V,\]
where the integral is with respect to a Haar measure $\mu_{G}$ on $G$ which we fix throughout. If $(\pi,V)$ is an admissible representation, it can be shown that the trace of the operator $\pi(f)$ is finite for all $f\in C^{\infty}_{c}(G)$. The resulting linear functional
\[\Theta_{\pi}: C^{\infty}_{c}(G)\longrightarrow \mathbb{C}\] given by
\[\Theta_{\pi}(f)=\tr(\pi(f))\]
is called the \textit{distribution character} of $\pi$. It determines the irreducible representation $\pi$ up to equivalence, i.e., if $\Theta_{\pi_{1}}(f)=\Theta_{\pi_{2}}(f)$, $\forall f\in C^{\infty}_{c}(G)$, then $\pi_{1}\simeq \pi_{2}$. \\

We now state a theorem of Harish-Chandra which is used to define the character of the representation $\pi$. \\

Let $G_{\reg}$ be the subset of regular semi-simple elements in $G$. It can be shown that $G_{\reg}$ is an open dense subset of $G$ whose complement has measure zero. The following is a deep result of Harish-Chandra.

\begin{theorem}[Harish-Chandra]\label{regularity theorem} There exists a locally integrable complex valued function $\Theta_{\pi}$ on $G$ such that $\Theta_{\pi}|_{G_{\reg}}$ is a locally constant function on $G_{\reg}$ and satisfies
\[\Theta_{\pi}(f)=\int_{G}f(g)\Theta_{\pi}(g)dg, \quad \forall f\in C^{\infty}_{c}(G).\]
\vspace{0.2 cm}\\
Also, for $x\in G_{\reg}, y\in G$, we have $$\Theta_{\pi}(yxy^{-1})=\Theta_{\pi}(x).$$
\end{theorem}
\vspace{0.2 cm}

We refer the reader to \cite{Har} for a proof of the above result. The locally constant function $\Theta_{\pi}$ on $G_{\reg}$ in the above theorem is called the \textit{character} of $\pi$. \\

Let $\widetilde{G}$ be a locally compact topological central extension of $G$ by $\mu_{n}$. Let $\xi: \mu_{n}\rightarrow \mathbb{C}^{\times}$ be any injective character of $\mu_{n}$. Let $(\pi,V)$ be a representation of $\widetilde{G}$. $\pi$ is called a \textit{genuine} representation, if for $\epsilon\in \mu_{n}$, $g\in \widetilde{G}$, we have
\[\pi(\epsilon g)=\xi(\epsilon)\pi(g).\]

 The above result of Harish-Chandra also holds in this setting. To be more precise, it can be shown that the distribution character of an irreducible admissible genuine representation $\pi$ of $\widetilde{G}$ is represented by a locally integrable function $\Theta_{\pi}$ on $\widetilde{G}$ which is locally constant on $\widetilde{G}_{\reg}=p^{-1}(G_{\reg})$ (here $p: \widetilde{G}\rightarrow G$ is the projection) and satisfies

\[\Theta_{\pi}(y^{-1}xy)=\Theta_{\pi}(x), \textrm{ for } x\in \widetilde{G}_{\reg}, y\in \widetilde{G}.\]
\vspace{0.1 cm}

We refer the reader to Theorem 4.3.2 and Corollary 4.3.3 in \cite{Wen} where results about character theory for metaplectic groups are discussed. See also Theorem I.5.1 in \cite{KazPat[2]} where the result is only mentioned. \\

Before we proceed further, we set up some notation and record a few properties of the character. Throughout we write $\mu=\mu_{\widetilde{G}}$ for the Haar measure on $\widetilde{G}$. For $K$ a compact open subgroup of $\widetilde{G}$, we let $e_K\in C_{c}^{\infty}(\widetilde{G})$ to be the  function 
\[e_K(g) = \begin{cases}
	\mu(K)^{-1}, & g \in K \\
	0, & g \notin K
\end{cases}.\] 

For $g \in \widetilde{G}$ and $f \in C_{c}^{\infty}(\widetilde{G})$, the natural action of $\widetilde{G}$ on $C_{c}^{\infty}(\widetilde{G})$ is denoted as $g.f$ where $$(g \cdot f)(x) := f(g^{-1}x).$$

\begin{lemma}\label{central-action}
		Let $(\pi, V)$ be an irreducible admissible genuine representation of $\widetilde{G}.$ Let $z$ be an element in the center of $\widetilde{G}$ and $\omega_\pi$ be the central character of $\pi$. Then, the following statements hold: 
	      \begin{enumerate}
		\item[\emph{1)}] $\pi(z \cdot f) = \omega_\pi(z) \pi(f),$ for all  $f \in C_{c}^{\infty}(\widetilde{G}).$ 
		\smallskip
		
		\item[\emph{2)}] $\Theta_{\pi}(z g) = \omega_\pi(z) \Theta_{\pi}( g),$ for all $g \in \widetilde{G}_{\reg}.$
 
	\end{enumerate}
\end{lemma}

\begin{proof}
	Let $f \in C_{c}^{\infty}(\widetilde{G})$. We have, 
	\begin{equation}\label{action-z-g}
		\pi(z \cdot f) = \int_{\widetilde{G}} (z \cdot f)(g) \pi(g) dg = \int_{\widetilde{G}} f(z^{-1}g) \pi(g) dg.
	\end{equation} 
By applying the change of variable $g \mapsto z g$ in \eqref{action-z-g} and using $\pi(z) = \omega_\pi(z) 1_{V},$ we get
$$\pi(z \cdot f) =  \omega_\pi(z)  \displaystyle\int_{\widetilde{G}} f(g) \pi(g) dg= \omega_\pi(z) \pi(f).$$ This proves 1). For $2)$, using the definition of the distribution character $\Theta_{\pi}$ we have,
\begin{equation*}%\label{evaluation-distribution-z.f}
	\Theta_{\pi}(z \cdot f) = \tr(\pi(z \cdot f)) = \omega_\pi(z) \tr(\pi(f)) = \omega_\pi(z) \Theta_{\pi}(f).	
\end{equation*}

\noindent Thus we get 
\begin{align*}
\Theta_{\pi}(z \cdot f) &= \omega_\pi(z) \Theta_{\pi}( f) \\
                        &=  \displaystyle\int_{\widetilde{G}}  f(g) \omega_\pi(z) \Theta_{\pi}(g) dg \\
                        &= \int_{\widetilde{G}} (z \cdot f)(g) \Theta_{\pi}(g) dg  \\
	&= \int_{\widetilde{G}} f(z^{-1}g) \Theta_{\pi}(g) dg \\
	&= \int_{\widetilde{G}} f(g) \Theta_{\pi}(zg) dg.
\end{align*} 
It follows that $$\Theta_{\pi}(z g) = \omega_\pi(z) \Theta_{\pi}( g).$$

\end{proof}

\begin{lemma}\label{distribution character-non-zero} Let $(\pi, V)$ be an irreducible admissible genuine representation of $\widetilde{G}.$ Then $$\Theta_{\pi}(f) \neq 0$$
for some $f\in C_{c}^{\infty}(\widetilde{G})$. 
\end{lemma}

\begin{proof} Let $0\neq u \in V$. Since $\pi$ is a smooth representation, there exists a compact open subgroup $K$ of $\widetilde{G}$ such that $u\in V^{K}$, where $V^{K} = \{v \in V \mid \pi(k)v = v, \forall k \in K\}$ is the subspace of $K$-fixed vectors in $V$. Let $f=e_{K}\in C_{c}^{\infty}(\widetilde{G})$. It is easy to see that $\pi(e_{K})\neq 0$ and defines a projection of $V$ onto $V^{K}$. Indeed, 
\[
	\pi(e_K)u = \int_{\widetilde{G}} e_K(g) \pi(g)u dg = \mu(K)^{-1} \int_{K} \pi(g)u dg = \mu(K)^{-1} \int_{K} u dg = u.
\]
It follows that $\pi(e_{K})$ fixes $V^{K}$ and $\pi(e_K) \neq 0.$ Let $v \in V$ such that $v \notin V^K$. Take $v_1 = \pi(e_K)v$. For $k \in K,$ we have
\begin{align*}
    \pi(k) v_1 &= \int_{\widetilde{G}} e_K(g)  \pi(kg)v dg \\
               &= \mu(K)^{-1} \int_{K}   \pi(kg)v dg \\
               &= \mu(K)^{-1} \displaystyle\int_{K}   \pi(g)v dg ~~~~~~~~~ \mbox{ ($g \mapsto k^{-1}g$)} \nonumber \\
               & = \displaystyle\int_{\widetilde{G}} e_K(g)  \pi(g)v dg \\
               & =  v_1.
\end{align*}
Hence $\pi(e_K)$ defines a projection of $V$ onto $V^K$. Thus we have $$\Theta_{\pi}(e_K) = \tr(\pi(e_K)) = \dim_{\mathbb{C}}(V^K)\neq 0.$$ Since $\pi$ is admissible, we have that $\dim_{\mathbb{C}}(V^K)$ is finite and the result follows. 
\end{proof}

\begin{lemma}\label{character-non-zero} Let $(\pi, V)$ be an irreducible admissible genuine representation of $\widetilde{G}.$ There exists $g_0 \in G_{\reg}$ such that $$\Theta_{\pi}((g_0, 1)) \neq 0.$$
\end{lemma}

\begin{proof} Suppose that $\Theta_{\pi}((g, 1)) = 0$ for all $g \in G_{\reg}$. For $x=(g,\eta)\in \widetilde{G}_{\reg},$ we have 
\[\Theta_{\pi}(x)=\Theta_{\pi}((g,\eta))= \omega_{\pi}(\eta)\Theta_{\pi}((g,1))=0. \]
Using Theorem~\ref{regularity theorem}, we have $\Theta_{\pi}(f)= 0$ for every $f \in C_{c}^{\infty}(\widetilde{G})$. This is not possible.
\end{proof}

\subsection{Some results about lifts of an automorphism of $G$ to $\widetilde{G}$}\label{results we need}

We recall a few results from \cite{Kab} which we need in proving our main result. Let $\widetilde{G}$ be a central extension of a group $G$ by an abelian group $A$. Let $p: \widetilde{G}\rightarrow G$ be the projection map, $s: G\rightarrow \widetilde{G}$ be a section of $p$ and $\tau$ be the 2-cocycle representing the class of this central extension in $\coh ^{2}(G,A)$ with respect to the section $s$. If $f: G\rightarrow G$ is an automorphism (resp. anti-automorphism) of $G$, then a lift of $f$ is an automorphism (resp. anti-automorphism) $\tilde{f}: \widetilde{G}\rightarrow \widetilde{G}$ such that \[p(\tilde{f}(g))=f(p(g)), \forall g\in \widetilde{G}.\]
Let $\mathcal{L}(f)$ denote the set of all lifts of $f$. The group $\Aut(G)$ acts on $\coh ^{2}(G,A)$ by $f[\sigma]= [\sigma\circ (f^{-1}\times f^{-1})]$ for any 2-cocycle $\sigma$.

\begin{proposition}\label{description of the set of liftings} The set $\mathcal{L}(f)$ is precisely described in terms of this action by the following:
\begin{enumerate}
\item[1)] The set $\mathcal{L}(f)$ is non-empty if and only if $f[\tau]=[\tau]$.
\item[2)] If $\mathcal{L}(f)$ is non-empty, then $\mathcal{L}(f)$ is a principal homogeneous space for the group $\Hom(G,A)$ under the action
\[(\phi.\tilde{f})(g)= \phi(p(g))\tilde{f}(g).\]
\end{enumerate}
\end{proposition}

We also need the following result (see Corollary 1 in \cite{Kab} for a proof) which discusses the continuity properties of the lift in the case when $G=\GL(m,F)$. We state it below for clarity.

\begin{proposition}\label{continuity of lift} Let $F$ be a non-Archimedean local field and suppose that the group of $n^{th}$ roots of unity in $F$ has order $n$. Let $\langle \, , \, \rangle$ be the $n^{th}$ order Hilbert symbol on $F$ and $\widetilde{\GL}(m)$ the corresponding metaplectic group. Then the lift of any topological automorphism of $\GL(m)$ to $\widetilde{\GL}(m)$ is also a topological automorphism.
\end{proposition}

\section{Lift of the standard involution}\label{sigma and its properties}

Let $G=\GL(2,F)$ and $\widetilde{G}=\widetilde{\GL}(2,F)$ be the $n$-fold metaplectic cover of $G$. Let $\tau$ be the standard involution on $G$. In this section, we define a lift $\sigma$ of 
$\tau$ and show that it is an involutive anti-automorphism of $\widetilde{G}$. We also discuss an important property of this lift (see Theorem~\ref{conjugacy property of sigma}) which is crucial in proving the main result of this paper. \\

For $\lambda\in F^{\times}$ and $g=\begin{bmatrix} a & b \\ c & d\end{bmatrix}\in G$, we let $u(\lambda)=\begin{bmatrix*}[r] \lambda & 0 \\ 0 & -\lambda\end{bmatrix*}$ and $\Delta(g)=\det(g)$. It is easy to see that
\[\tau(g)= w_{0}g^{\top}w_{0}= u(\Delta(g))g^{-1}u(1).\]

For $\epsilon\in \mu_{n}$, we denote the element $(I_{2},\epsilon)\in \widetilde{G}$ as $\epsilon$. Let $\tilde{u}(\lambda)=(u(\lambda),1)$ and $\tilde{z}(\lambda)= (\lambda I_{2}, 1)$ where $I_{2}$ is the $2\times 2$ identity matrix. We extend $\Delta$ to $\widetilde{G}$ by $\Delta((g,\epsilon))=\Delta(g)$.
For $h=(g,\epsilon)\in \widetilde{G}$ define
\begin{equation}\label{definition of sigma}\sigma(h)=\tilde{u}(\Delta(h))h^{-1}\tilde{u}(1).\end{equation}

Before discussing the properties of $\sigma$, we first prove a proposition which is used to study the properties of 
$\sigma$.  

\begin{proposition}\label{properties to check sigma is an involution} Let $\lambda, \lambda_1, \lambda_2 \in F^{\times}$, and $h\in \widetilde{G}$. Then
\begin{enumerate}[label=\emph{(\roman*)}]
\item $h\tilde{z}(\lambda)= \langle \Delta(h), \lambda\rangle\; \tilde{z}(\lambda)h$.
\vspace{0.1 cm}
\item $\tilde{z}(\lambda_{1})\tilde{z}(\lambda_{2})= \langle \lambda_{1}, \lambda_{2}\rangle \tilde{z}(\lambda_{1}\lambda_{2})$.
\vspace{0.1 cm}
\item $\tilde{u}(\lambda_{1})\tilde{u}(\lambda_{2})= \langle \lambda_{1}, -\lambda_{2}\rangle \tilde{z}(\lambda_{1}\lambda_{2})$.
\vspace{0.1 cm}
\item $\tilde{u}(\lambda)^{-1}= \tilde{u}(\lambda^{-1})$.
\vspace{0.1 cm}
\item $\tilde{u}(\lambda_{1})\tilde{z}(\lambda_{2})= \langle\lambda_{1}, \lambda_{2}\rangle\tilde{u}(\lambda_{1}\lambda_{2})$.
\end{enumerate}
\end{proposition}

\begin{proof} \hfill \\ 
\begin{enumerate}[label=(\roman*)]
\item Let $h=\left(\begin{bmatrix*}[r] a & b\\ c & d\end{bmatrix*}, \epsilon\right)$. It is easy to see that \[\tilde{z}(\lambda)h=\Bigg(\begin{bmatrix*}[r] a\lambda & b\lambda \\ c\lambda & d\lambda\end{bmatrix*}, c\left (\begin{bmatrix*}[r] \lambda & 0 \\ 0 & \lambda\end{bmatrix*}, \begin{bmatrix} a & b \\ c & d\end{bmatrix} \right )\epsilon \Bigg),\] and \[h\tilde{z}(\lambda)=\Bigg(\begin{bmatrix*}[r] a\lambda & b\lambda \\ c\lambda & d\lambda\end{bmatrix*}, c\left ( \begin{bmatrix} a & b \\ c & d\end{bmatrix}, \begin{bmatrix*}[r] \lambda & 0 \\ 0 & \lambda\end{bmatrix*} \right )\epsilon \Bigg).\]
Note that $c\left (\begin{bmatrix*}[r] \lambda & 0 \\ 0 & \lambda\end{bmatrix*}, \begin{bmatrix} a & b \\ c & d\end{bmatrix} \right )=\begin{cases} \langle \lambda, c\rangle , & c\neq 0 \\
\langle \lambda, d\rangle,  & c=0.\end{cases}$\\
Therefore
\begin{align*}
 \langle \Delta(h), \lambda\rangle c\left (\begin{bmatrix*}[r] \lambda & 0 \\ 0 & \lambda\end{bmatrix*}, \begin{bmatrix} a & b \\ c & d\end{bmatrix} \right )&=\begin{cases} \langle \Delta(h), \lambda\rangle \langle \lambda, c\rangle , & c\neq 0 \\
\langle \Delta(h), \lambda\rangle \langle \lambda, d\rangle,  & c=0\end{cases}\\
&=\begin{cases}
    \langle \lambda, \Delta(h)^{-1}c\rangle, & c\neq 0\\
    \langle \lambda, \Delta(h)^{-1}d\rangle, & c= 0
\end{cases}\\
&=c\left (\begin{bmatrix} a & b \\ c & d\end{bmatrix}, \begin{bmatrix*}[r] \lambda & 0 \\ 0 & \lambda\end{bmatrix*} \right ).
\end{align*}
The result follows. \\

\item  Since $c(z(\lambda_1), z(\lambda_2))=c(\lambda_1 I_2, \lambda_2 I_2)=\langle \lambda_1, \lambda_2\rangle$, we have \[\tilde{z}(\lambda_{1})\tilde{z}(\lambda_{2})= (\lambda_1 \lambda_2 I_2,\langle \lambda_1, \lambda_2\rangle )= \langle \lambda_{1}, \lambda_{2}\rangle \tilde{z}(\lambda_{1}\lambda_{2}).\]

\item Using \[c(u(\lambda_1),u(\lambda_2))=c\left (\begin{bmatrix*}[r] \lambda_1 & 0 \\ 0 & -\lambda_1\end{bmatrix*}, \begin{bmatrix} \lambda_2 & 0 \\ 0 & -\lambda_2 \end{bmatrix} \right )= \langle \lambda_1, -\lambda_2\rangle,\] we get \[\tilde{u}(\lambda_{1})\tilde{u}(\lambda_{2})= (\lambda_{1}\lambda_{2} I_2, \langle \lambda_1, -\lambda_2\rangle)= \langle \lambda_{1}, -\lambda_{2}\rangle \tilde{z}(\lambda_{1}\lambda_{2}).\] 

\item We have $\tilde{u}(\lambda)=(u(\lambda),1)$. Thus, 
\begin{align*}
    \tilde{u}(\lambda)^{-1}&=\big(u(\lambda)^{-1},c(u(\lambda),u(\lambda)^{-1})^{-1}\big)\\
&= \big(u(\lambda)^{-1}, c(u(\lambda),u(\lambda^{-1}))^{-1}\big)\\
    &=\big(u(\lambda^{-1}),\langle -\lambda, \lambda \rangle^{-1}\big)\\
    &= \big(u(\lambda^{-1}),1\big)\\
    &=\tilde{u}(\lambda^{-1}).\\
\end{align*}

\item It is easy to see that, 
\[c(u(\lambda_1),z(\lambda_2))=c\left (\begin{bmatrix*}[r] \lambda_1 & 0 \\ 0 & -\lambda_1\end{bmatrix*}, \begin{bmatrix} \lambda_2 & 0 \\ 0 & \lambda_2 \end{bmatrix} \right )= \langle \lambda_2, \lambda_1^{-1}\rangle=\langle \lambda_1, \lambda_2\rangle.\]
Thus we have, 
\[\tilde{u}(\lambda_{1})\tilde{z}(\lambda_{2})= \Bigg(\begin{bmatrix*}[r] \lambda_1\lambda_2 & 0 \\ 0 & -\lambda_1\lambda_2\end{bmatrix*}, \langle\lambda_{1}, \lambda_{2}\rangle\Bigg)=\langle\lambda_{1}, \lambda_{2}\rangle\tilde{u}(\lambda_{1}\lambda_{2}).\]
\end{enumerate}
\end{proof}

Let $\sigma: \widetilde{G}\rightarrow \widetilde{G}$ be defined as in \eqref{definition of sigma}. We record some properties satisfied by $\sigma$ in the proposition below. 

\begin{lemma}\label{Properties of sigma} The map $\sigma$ satisfies the following properties:
\begin{enumerate}
    \item[\emph{(1)}] $\sigma$ is an anti-automorphism.
    \vspace{0.1 cm}
    \item[\emph{(2)}] $\sigma(\epsilon)=\epsilon^{-1}$ for any $\epsilon\in \mu_{n}$.
    \vspace{0.1 cm}
    \item[\emph{(3)}] $\sigma$ is an involution.
    \vspace{0.1 cm}
    \item[\emph{(4)}] $\sigma(h^{-1})=\sigma(h)^{-1}$ for all $h\in \widetilde{G}$.
    \vspace{0.1 cm}
    \item[\emph{(5)}] $\sigma$ is a lift of $\tau$.
    \end{enumerate}
\end{lemma}

\begin{proof} \hfill \\
\begin{enumerate}
\item We use properties (i), (iii), and (v) from Proposition \ref{properties to check sigma is an involution}. For $h_{1}, h_{2}\in \widetilde{G}$, we have \begin{align*}
\sigma(h_{1}h_{2}) &= \tilde{u}(\Delta(h_{1}h_{2}))(h_{1}h_{2})^{-1}\tilde{u}(1)\\
&= \tilde{u}(\Delta(h_{2})\Delta(h_{1}))(h_{1}h_{2})^{-1}\tilde{u}(1)\\
&= \left<\Delta(h_{1}), \Delta(h_{2})\right>\tilde{u}(\Delta(h_{2}))\tilde{z}(\Delta(h_{1}))h_{2}^{-1}h_{1}^{-1}\tilde{u}(1)\\
&= \tilde{u}(\Delta(h_{2}))h_{2}^{-1}\tilde{z}(\Delta(h_{1}))h_{1}^{-1}\tilde{u}(1)\\
&= \tilde{u}(\Delta(h_{2}))h_{2}^{-1}\tilde{u}(1)\tilde{u}(\Delta(h_{1}))h_{1}^{-1}\tilde{u}(1)\\
&= \sigma(h_{2})\sigma(h_{1}).
\end{align*}

\item Let $1\neq \epsilon\in \mu_{n}$ be any non-trivial element in $\mu_{n}$. Identifying $\epsilon$ with ($I_2,\epsilon$), we obtain $\Delta(\epsilon)=1$. We have
\begin{align*}
\sigma(\epsilon) &= \tilde{u}(\Delta(\epsilon))\epsilon^{-1}\tilde{u}(1)\\
&=(u(\Delta(\epsilon)), 1)(I_2,\epsilon)^{-1}(u(1),1)\\
&= \left (u(1), 1 \right )\left (I_2, \epsilon^{-1} \right )\left (u(1), 1 \right )\\
&= \left (u(1), \epsilon^{-1} \right )\left (u(1), 1 \right) & \text{(since $c\left (u(1), I_2 \right )=1$)}\\
&= (I_2, \epsilon^{-1}) & \text{(since $c\left (u(1), u(1) \right )=1$)}\\
&= \epsilon^{-1}.
\end{align*}

\item  For $h=(g, \epsilon)\in \widetilde{G}$, observe that $\Delta(\sigma(h))=\Delta(u(\Delta(g))g^{-1}u(1))= \Delta(h)$. Using the properties (i)-(iv) mentioned in Proposition~\ref{properties to check sigma is an involution}, we obtain
\begin{align*}
\sigma((\sigma(h))) &= \sigma( \tilde{u}(\Delta(h))h^{-1}\tilde{u}(1))\\
&= \tilde{u}(\Delta(h))\tilde{u}(1)h\tilde{u}(\Delta(h)^{-1})\tilde{u}(1)\\
&= \langle \Delta(h), -1\rangle \tilde{z}(\Delta(h))h\langle \Delta(h)^{-1}, -1\rangle \tilde{z}(\Delta(h)^{-1})\\
&= \langle \Delta(h), \Delta(h)\rangle h \tilde{z}(\Delta(h))\tilde{z}(\Delta(h)^{-1})\\
&= \langle \Delta(h), \Delta(h)\rangle h\langle \Delta(h), \Delta(h)^{-1}\rangle \\
&= h.
\end{align*}

\item It follows directly from the fact that $\sigma$ is an anti-automorphism. \\

\item It is enough to consider the case when $h=(g,1)$, with $g\in G$. Indeed, if we have $\sigma(g,1)=(x,\xi)$, then $\sigma(g,\epsilon) = (x, \epsilon^{-1}\xi)$, using (1) and (2). This implies that $x= (p\circ \sigma)(g,\epsilon) = (p\circ\sigma)(g,1)$. To show that $\sigma$ is a lift of $\tau$, it is enough to show that $(p\circ \sigma)(h)=  (\tau\circ p)(h)$ for $h=(g, 1) \in \widetilde{G}$.  
We now evaluate $\sigma$, depending on whether $c$ is non-zero or zero. If $c\neq 0$, we have
\begin{align*}
\sigma(h) &=\sigma(g,1)\\ 
&= (u(\Delta(g)), 1)(g^{-1}, 1)(u(1), 1)\\
&= (u(\Delta(g))g^{-1}, \left<\Delta(g), c\right>)(u(1), 1)\\
&= (u(\Delta(g)g^{-1}u(1), \left<\Delta(g), c\right>)\\
&= (\tau(g), \left<\Delta(g), c\right>),
\end{align*}
and if $c=0$, we have
\begin{align*}
\sigma(h) &=\sigma(g,1)\\ 
&=(u(ad), 1)(g^{-1}, \left<a, d\right>)(u(1), 1)\\
&= (u(ad)g^{-1}, \left<d, a\right>\left<d, d\right>\left<a, d\right>)(u(1), 1)\\
&= (u(ad)g^{-1}, \left<d, d\right>)(u(1), 1)\\
&= (u(ad)g^{-1}u(1), \left<d, d\right>\left<d, -1\right>)\\
&= (\tau(g), 1).
\end{align*}
In both the cases, we see that $(p\circ \sigma)(h)= \tau(g)= (\tau\circ p)(h)$. Hence the result follows.

\end{enumerate}
\end{proof}

\begin{remark}\label{sigma-h-computation} We have used the following computations in part (5) of Lemma \ref{Properties of sigma}. Let $g = \begin{bmatrix}
		a & b \\
		c & d
	\end{bmatrix}\in G$ and $h = (g, 1)\in \widetilde{G}$. Then,

\begin{enumerate}
\item[(1)] $c(g, g^{-1}) = c(g^{-1}, g)=  \begin{cases}
		         1, &   \mbox{if $c \neq 0$} \\
		\langle d, a \rangle, &  \mbox{if $c =0$}
        \end{cases}$.
\vspace{0.1 cm}\\

\item[(2)] $h ^{-1} = (g^{-1}, c(g, g^{-1})^{-1} ) = \begin{cases}
	(g^{-1}, 1), &   \mbox{if $c \neq 0$} \\
	(g^{-1}, \langle a, d \rangle), &  \mbox{if $c =0$}
\end{cases}$. 
\vspace{0.1 cm}\\

\item[(3)] $\sigma(h) =  \begin{cases}
	(\tau(g), \langle \triangle(h), c \rangle ), &   \mbox{if $c \neq 0$} \\
	(\tau(g), 1 ), &  \mbox{if $c=0$}
\end{cases}$.
\end{enumerate} 
\end{remark}

The involution $\sigma$ satisfies a certain conjugation property which plays an important role in the proof of our main result. Before we proceed further, we record a few more lemmas that we need to establish this property of $\sigma$. \\

\begin{lemma}\label{properties-cocycles-used-proof-Gtilde-equal-S}
For $x, g\in G$, let \[\lambda =  c(xg, x^{-1}) c(x, g) c(x, x^{-1})^{-1},\] and \[\beta =  c(x, x^{-1}) c(g^{-1} x ^{-1}, x) c(g, x^{-1})^{-1}  c(x, g x^{-1})^{-1} c(x, g^{-1} x^{-1})^{-1}.\] Then,
\begin{enumerate}
\item[\emph{(1)}]  $c(x g^{-1} x^{-1}, x g x ^{-1}) = c(g^{-1}, g) \beta.$
		\smallskip
		
\item[\emph{(2)}] $\lambda \beta = c(g^{-1} x^{-1}, x) c(x, g^{-1} x ^{-1})^{-1}.$
	\end{enumerate} 
\end{lemma}

\begin{proof} Using the cocycle property, we have for $g_{1}, g_{2}, g_{3}\in G$
    \begin{equation}\label{associativity-cocycles}
        c(g_1 g_2, g_3) c(g_1, g_2) = c(g_1, g_2 g_3) c(g_2, g_3).
    \end{equation}

For (1), taking $g_1 = x, g_2 = g^{-1} x^{-1}$ and $g_3 = xg x^{-1}$ in \eqref{associativity-cocycles}, we get   
     \begin{align}\label{properties-cocycles-used-proof-Gtilde-equal-S-eq1}
        c(x g^{-1} x^{-1}, x g x ^{-1}) = & c(x, x^{-1}) c(g^{-1} x^{-1}, xg x^{-1}) c(x, g^{-1} x^{-1})^{-1}.
               \end{align}
Applying \eqref{associativity-cocycles} again  with $g_1 = g^{-1} x^{-1}, g_2 = x$ and $g_3 = g x^{-1}$ and then again with $g_1 = g^{-1}, g_2 = g$ and $g_3 = x^{-1}$ we get the following equalities:
        \begin{align}\label{properties-cocycles-used-proof-Gtilde-equal-S-eq2}
        c(g^{-1} x^{-1}, xg x^{-1}) = & c(g^{-1} , g x^{-1}) c(g^{-1} x^{-1}, x) c(x, g x^{-1})^{-1} \nonumber 
        \vspace{0.1 cm}\\
        = & [ c(g^{-1}, g) c( g,  x^{-1})^{-1}]  c(g^{-1} x^{-1}, x) c(x, g x^{-1})^{-1}.
    \end{align}

   Now, (1) follows by \eqref{properties-cocycles-used-proof-Gtilde-equal-S-eq1} and \eqref{properties-cocycles-used-proof-Gtilde-equal-S-eq2}.  \\

 Proof of (2) is a direct application of \eqref{associativity-cocycles}  with $g_1 = x, g_2 = g$ and $g_3 = x^{-1}.$ 
\end{proof}

\begin{lemma}\label{g-conjugate-g-2 or g-3} Let $g = \begin{bmatrix}
		\alpha & \beta \\
		0 & \delta
	   \end{bmatrix} \in G.$ Then, $g$ is conjugate to either $\begin{bmatrix}
		\alpha & 0 \\
		0 & \delta
	\end{bmatrix}$ or $\begin{bmatrix}
		\alpha & 1 \\
		0 & \alpha
	\end{bmatrix}.$ 
\end{lemma}

\begin{proof}
	If $\alpha \neq \delta,$ taking $x = \left[\begin{array}{cc}
		1 & \beta/(\alpha- \delta) \\
		0 & -1
	\end{array}\right]$ we get $x g x^{-1} = \begin{bmatrix}
		\alpha & 0 \\
		0 & \delta
	\end{bmatrix}.$ Suppose that $\alpha = \delta.$ If $\beta =0$, then we have $g = \alpha I_2.$ On the other hand, if $\beta \neq 0$, choosing $x = \begin{bmatrix}
		1 & 0 \\
		0 & \beta
	\end{bmatrix}$ we have $x g x^{-1} = \begin{bmatrix}
		\alpha & 1 \\
		0 & \alpha
	\end{bmatrix}.$
\end{proof} 
Let $\sigma$ be defined as in \eqref{definition of sigma}.  Consider the set 
\begin{equation}\label{Definition of the set S}
S:= \{h = (g, \eta) \in \widetilde{G}: \sigma(h) = c(g^{-1}, g)^{-2} \eta^{-2} x h x^{-1}, \mbox{for some $x \in \widetilde{G}$}\}.
\end{equation}

\begin{lemma}\label{closed-under-epsilon}
 Let  $h \in S$ and $\epsilon\in \mu_{n}$. Then $$\epsilon h \in S.$$
 That is, $S$ is closed under multiplication by $\epsilon\in \mu_{n}$. 
\end{lemma}

\begin{proof} Consider $e = (I_2, 1)\in \widetilde{G}$. Since $c(I_2, I_2) =1$, we have $e \in S$. Thus, $S$ is non-empty. Let $h = (g, \eta) \in S.$ Choose $z \in \widetilde{G}$ such that  $\sigma(h) = c(g^{-1}, g)^{-2} \eta^{-2} z h z^{-1}.$  Then, we have
\begin{align*}
\sigma (\epsilon h) &= \epsilon^{-1} \sigma(h) 
\vspace{0.1 cm}\\
&= \epsilon^{-1} \left[ c(g^{-1}, g)^{-2} \eta^{-2} z h z^{-1}\right] 
\vspace{0.1 cm}\\
&=  c(g^{-1}, g)^{-2} (\eta \epsilon)^{-2} z (\epsilon h) z^{-1}.
\end{align*} 
\end{proof}

\begin{lemma}\label{h_i-arein-S} Consider the matrices $g_{1}, g_{2}, g_{3}$ in $G$ defined as follows: \\
\[g_1 = \begin{bmatrix}
		0 & v \\
		1 & w
	\end{bmatrix}: v\in F^{\times}, w\in F,\,g_2 = \begin{bmatrix}
		a & 0 \\
		0 & d
	\end{bmatrix}: a \neq d, a,d\in F^{\times} \textrm{ and } g_3 = \begin{bmatrix}
		b & 1 \\
		0 & b
	\end{bmatrix}:b\in F^{\times}.\] For $\epsilon \in \mu_n,$ and $i \in \{1,2,3\}$, we have $(g_i, \epsilon) \in S$. 
\end{lemma}

\begin{proof}
For 	$i \in \{1,2,3\},$ write $h_i = (g_i, 1).$  By Lemma \ref{closed-under-epsilon}, it is enough to show that $h_i \in S.$
Consider $h_{1}=(g_{1},1)$. Let $y= \left[\begin{array}{cc}
	1 & 0 \\
	-w/v & 1
\end{array}\right]$ and $\widetilde{y} = (y, 1).$ It is easy to see that \[\tau(g_1) = \begin{bmatrix}
	w & v \\
	1 & 0
\end{bmatrix} = y g_1 y^{-1}.\]  
A simple calculation gives $c(y g_1, y^{-1})=1$ and $c(y, g_1)=1.$ Since $\langle \triangle(h_1), 1 \rangle = c(g_1^{-1}, g_1)= 1$, using (3) of Remark \ref{sigma-h-computation}, it follows that $$\widetilde{y} h_1 \widetilde{y}^{-1} = (\tau(g_1), 1)=\sigma(h_1)$$ and $h_1\in S$. \\

Consider $h_{2}=(g_{2},1)$. Let $z= \begin{bmatrix}
	0 & d \\
	1 & 0
\end{bmatrix}$ and set $\widetilde{z} = (z, 1).$ Clearly we have \[z g_2 z^{-1} = \begin{bmatrix}
	d & 0 \\
	0 & a
\end{bmatrix} = \tau(g_2).\]
It is easy to see that $c(z g_2, z^{-1}) = c(z , z^{-1})=1$, $c( g_2^{-1}, g_2)= \langle d, a \rangle$  and $c(z, g_2) = \langle d, a \rangle ^{2}$. Using (3) of Remark \ref{sigma-h-computation}, we get 
$$\widetilde{z} h_2 \widetilde{z}^{-1} = (\tau(g_2), \langle d, a \rangle ^2 )= c( g_2^{-1}, g_2)^{2}(\tau(g_2), 1)= c( g_2^{-1}, g_2)^{2}\sigma(h_{2}).$$ Thus $h_{2}\in S$. \\

We now show that $h_3=(g_{3},1) \in S$. It is clear that  $\tau(g_3) = g_3$. By (3) of Remark \ref{sigma-h-computation}, we have $\sigma(h_3)  = (g_3, 1)$. Since $c( g_3^{-1}, g_3)=  \langle b, b \rangle$ and $\langle b, b \rangle^2 =1,$ we get $\sigma(h_3) = c( g_3^{-1}, g_3)^{-2}  h_3.$ 
\end{proof}

\begin{lemma}\label{working-equation-prove-Gtilde-equaltoS}
Let $g_1,$ $g_2$ and $g_3$ be as in Lemma~\ref{h_i-arein-S}. Let $g \in G$, and let $h = (g,1)\in \widetilde{G}$. Suppose that there exists $x\in G$ such that $x g x^{-1} = g_i$, for some $i \in \{1,2,3\}$. Let $A = c(x^{-1} g_i^{-1}, x) c(x, x^{-1}g_i^{-1})^{-1}$. Then, for  $s \in F^{\times}$ there exists $z \in \widetilde{G}$ such that \[\sigma(h) =  \langle s, \triangle(h) \rangle A^{-2}  c(g^{-1}, g)^{-2} z h z^{-1}.\] 
\end{lemma}

\begin{proof} Let $\widetilde{x} = (x, 1)$ and  $\lambda = c(xg, x^{-1}) c(x, g) c(x, x^{-1})^{-1}$. Since $xgx^{-1}=g_{i}$, we have
\begin{equation}\label{Gtilde-equal-S-eq1}
\widetilde{x} h \widetilde{x}^{-1} = (xg x^{-1}, \lambda) = (g_i, \lambda).
\end{equation}
Let $s \in F^{\times}.$ Using (i) of Proposition~\ref{properties to check sigma is an involution}, we have 
	\begin{align}\label{Gtilde-equal-S-eq2}
		(g_i, \lambda) &= \widetilde{x} \widetilde{z}(s)^{-1} \widetilde{z}(s) h \widetilde{z}(s)^{-1} \widetilde{z}(s) \widetilde{x}^{-1} \nonumber 
        \vspace{0.1 cm} \\
  		&= \langle s, \triangle(h) \rangle \widetilde{x} \widetilde{z}(s)^{-1} h \widetilde{z}(s) \widetilde{x}^{-1} \nonumber 
        \vspace{0.1 cm}\\
        &= \langle s, \triangle(h) \rangle u h u^{-1}; \mbox{ where $u = \widetilde{x} \widetilde{z}(s)^{-1}$.}
	\end{align}
	
 By Lemma \ref{h_i-arein-S}, $(g_i, \lambda) \in S.$  Choose $y \in \widetilde{G}$ such that
	\begin{equation}\label{Gtilde-equal-S-eq4}
		\sigma((g_i, \lambda)) = c(g_i^{-1}, g_i)^{-2} \lambda^{-2} y (g_i, \lambda) y^{-1}.
	\end{equation}  
Applying $\sigma$ on both sides to \eqref{Gtilde-equal-S-eq2}, and using Lemma \ref{Properties of sigma} and \eqref{Gtilde-equal-S-eq4}, we get 
\begin{equation}\label{Gtilde-equal-S-eq5}
		c(g_i^{-1}, g_i)^{-2} \lambda^{-2} y (g_i, \lambda) y^{-1} = \langle s, \triangle(h) \rangle^{-1} \sigma(u)^{-1} \sigma(h) \sigma(u).
\end{equation}
Simplifying \eqref{Gtilde-equal-S-eq5}, we have
	\begin{align}\label{Gtilde-equal-S-eq6}
		\sigma(h) = & \langle s, \triangle(h) \rangle c(g_i^{-1}, g_i)^{-2} \lambda^{-2} \sigma(u) y (g_i, \lambda) y^{-1} \sigma(u)^{-1} \nonumber  \vspace{0.1 cm}  \\
		= & \langle s, \triangle(h) \rangle c(g_i^{-1}, g_i)^{-2} \lambda^{-2} \sigma(u) y \widetilde{x}  \left[\widetilde{x}^{-1}(g_i, \lambda) \widetilde{x}\right] \widetilde{x}^{-1} y^{-1} \sigma(u)^{-1}.
	\end{align}
	Write $z = \sigma(u) y \widetilde{x}.$ By \eqref{Gtilde-equal-S-eq1},  $\widetilde{x}^{-1}(g_i, \lambda) \widetilde{x} = h.$ Thus, \eqref{Gtilde-equal-S-eq6} gives:
	\begin{equation*}\label{Gtilde-equal-S-eq7}
		\sigma(h) =  \langle s, \triangle(h) \rangle c(g_i^{-1}, g_i)^{-2} \lambda^{-2} z h z^{-1}.
	\end{equation*}
	
Applying Lemma~\ref{properties-cocycles-used-proof-Gtilde-equal-S}, the result follows.
\end{proof}

\begin{theorem}\label{conjugacy property of sigma} For $h = (g, \eta) \in \widetilde{G},$ we have $$\sigma(h) = c(g^{-1}, g)^{-2} \eta^{-2} z h z^{-1},$$ for some $z \in \widetilde{G}.$ %In other words, $\widetilde{G} = S.$
\end{theorem}

\begin{proof} Let $g \in G$ and set $h = (g,1).$	By Lemma \ref{closed-under-epsilon}, it is enough to show that $h \in S$. Suppose that $g = \begin{bmatrix}
\alpha & 0 \\
0 & \alpha
\end{bmatrix}$, for some $\alpha \in F^{\times}.$ Take $u= \begin{bmatrix}
0 & \alpha \\
1 & 0
\end{bmatrix}$, and write $\widetilde{u} = (u,1).$ Since $\tau(g) = g,$ by (3) of Remark \ref{sigma-h-computation}, we get  $\sigma(h) = h.$   It is easy to see that  $c(u g, u^{-1}) = 1= c(u , u^{-1})$ and  $c(u, g) = \langle \alpha, \alpha \rangle^2 =1.$ Hence, $\widetilde{u} h \widetilde{u}^{-1} = h.$ Observe that $c(g^{-1} , g) = \langle \alpha, \alpha \rangle.$  Thus,    $$\sigma(h) = c(g^{-1} , g)^{-2} \widetilde{u} h \widetilde{u}^{-1}.$$
Suppose that $g$ is not a scalar matrix. %In this case also, we will show that  $h=(g,1) \in S.$
Let $g_1,$ $g_2$ and $g_3$ be as in Lemma \ref{h_i-arein-S}.  Using the rational canonical form, it follows that $g$ is conjugate to a matrix of the type either $g_1,$ $g_2,$ or $g_3$. We discuss these cases separately. 

\textbf{Case (1)}:   Assume that $g$ is conjugate to $g_1,$ where $g_1 = \begin{bmatrix}
0 & v \\
1 & w
\end{bmatrix}$.    Suppose $g = \begin{bmatrix}
	\alpha & \beta \\
	\gamma & \delta
\end{bmatrix}$.  Since $g$ is conjugate to $g_1,$ we have $$v = - \Delta(g) = -\det(g) \quad \text{and} \quad w = \tr(g) = \alpha + \delta.$$ 
Also, $\gamma \neq 0$ by Lemma \ref{g-conjugate-g-2 or g-3}. Consider $x = \begin{bmatrix}
0 & v \gamma^{-1} \\
1 & \delta \gamma^{-1}
\end{bmatrix}$. It is easy to verify that $x g x^{-1} = g_1.$ By Lemma \ref{working-equation-prove-Gtilde-equaltoS},  there exists $z \in \widetilde{G}$ such that 
\begin{equation}\label{working-equatioto-provethe-theorem-1}
\sigma(h) =   A^{-2}  c(g^{-1}, g)^{-2} z h z^{-1},
\end{equation} 
 where $A = c(x^{-1} g_1^{-1}, x) c(x, x^{-1}g_1^{-1})^{-1}.$ We now compute $c(x^{-1} g_1^{-1}, x)$ and  $c(x, x^{-1}g_1^{-1}).$ These are as follows:
\begin{equation*}
	c(x^{-1} g_1^{-1}, x) = \begin{cases}
		\langle 1, -\triangle(g) \rangle, & w=0 \\
		\langle w^{-1}, -\triangle(g) \rangle, & w \neq 0
	\end{cases} \quad \text{and} \quad c(x, x^{-1}g_1^{-1}) = \begin{cases}
		\langle -\triangle(g)^{-1},  \triangle(g)^{-1} \rangle, & w=0 \\
		\langle -\triangle(g), w \rangle, & w \neq 0.
	\end{cases}
\end{equation*}
This implies that  $c(x^{-1} g_1^{-1}, x) c(x, x^{-1}g_1^{-1})^{-1} = 1$ irrespective of $w$ is zero or non-zero.
Hence, $A =1.$ Thus, \eqref{working-equatioto-provethe-theorem-1} reduces to $$ \sigma(h) =    c(g^{-1}, g)^{-2} z h z^{-1}.$$ 

\textbf{Case (2)}:  Assume that $g$ is conjugate to $g_2,$ where $g_2 = \left[\begin{array}{cc}
a & 0\\
0 & d
\end{array}\right]$ with $a \neq d.$ Choose  $x \in G$ such that $x g x^{-1} = g_2.$  We claim the following:  % If $A^{-2} =1,$ by taking $s =1$ in \eqref{working-equatioto-provethe-theorem-2}, we get $\sigma(h) =   z h z^{-1}.$ On the other-hand, if  $A^{-2} \neq 1,$
\begin{equation}\label{g-conjugate-g_2}
\mbox{There exists $y \in G$ such that $y g y^{-1} = g_2$ and  $c(y^{-1} g_2^{-1}, y) c(y, y^{-1}g_2^{-1})^{-1} =1.$}
\end{equation}
Suppose that \eqref{g-conjugate-g_2} is true. Hence, in this case result follows by applying Lemma \ref{working-equation-prove-Gtilde-equaltoS} with $y g y^{-1} = g_2.$ 
 We now prove the claim \eqref{g-conjugate-g_2}. Write  $x = \begin{bmatrix}
f & p \\
q & r
\end{bmatrix}$, and consider the following cases:\\

(i) Suppose that $q =0.$ Take $y = tx,$ where $t= \begin{bmatrix}
	f^{-1} & 0\\
	0 & r^{-1}
\end{bmatrix}$. Clearly, $y g y^{-1} = g_2.$  Also, we have	$c(y^{-1} g_2^{-1}, y) = 	\langle 1, a \rangle$ and $c(y, y^{-1}g_2^{-1}) = \langle d, 1 \rangle.$ Thus, $$c(y^{-1} g_2^{-1}, y) c(y, y^{-1}g_2^{-1})^{-1} =1.$$ 

(ii) Suppose that $q \neq 0.$ Let $y = tx,$ where $t= \begin{bmatrix}
	p^{-1} & 0\\
	0 & q^{-1} d^{-1}
\end{bmatrix}$ if $f =0$ and  $t= \begin{bmatrix}
\frac{d}{(d-a)f} & 0\\
0 & q^{-1} d^{-1}
\end{bmatrix}$ if $f \neq 0.$ It is clear that  $y g y^{-1} = g_2.$ Also, 	$c(y^{-1} g_2^{-1}, y) = 	\langle 1, -d \rangle$ and $c(y, y^{-1}g_2^{-1}) = \langle 1, -a \rangle.$ Thus,  $$c(y^{-1} g_2^{-1}, y) c(y, y^{-1}g_2^{-1})^{-1} =1.$$

\textbf{Case (3)}:  Assume that $g$ is conjugate to $g_3,$ where $g_3 = \begin{bmatrix}
	b & 1\\
	0 & b
\end{bmatrix}$. Let  $x \in G$ such that $x g x^{-1} = g_3.$ 
Given  $s \in F^{\times},$ there exists $z \in \widetilde{G}$ (see Lemma \ref{working-equation-prove-Gtilde-equaltoS}) such that 
\begin{equation}\label{working-equatioto-provethe-theorem-3}
	\sigma(h) =  \langle s, \triangle(h) \rangle A^{-2} c(g^{-1}, g)^{-2} z h z^{-1}.
\end{equation} 
 Here, $A = c(x^{-1} g_3^{-1}, x) c(x, x^{-1}g_3^{-1})^{-1}.$ Write  $x = \begin{bmatrix}
	f & p \\
	q & r
\end{bmatrix}$. 
Then, we have \begin{equation*}
	c(x^{-1} g_3^{-1}, x) = \begin{cases}
		\langle r, fb \rangle, & q=0 \\
		\langle q, b \rangle, & q \neq 0
	\end{cases} \quad \text{and} \quad c(x, x^{-1}g_3^{-1}) = \begin{cases}
		\langle br, f \rangle, & q=0 \\
		 \langle b, -q \rangle, & q \neq 0.
	\end{cases}
\end{equation*}

Thus, 
 \begin{equation*}
	A = \begin{cases}
		\langle fr, b \rangle, & q=0 \\
		\langle -q^2, b \rangle, & q \neq 0.
	\end{cases}
	\end{equation*}
Note that  $\triangle(h) = b^2.$ If $q=0,$ taking $s = fr$ we get  $\langle s, \triangle(h) \rangle A^{-2} =1.$ Thus,
$$\sigma(h) = c(g^{-1}, g)^{-2} z h z^{-1},$$  for some $z \in \widetilde{G}$ via \eqref{working-equatioto-provethe-theorem-3}. Also, $s = - q^2$ works when $q \neq 0.$  Proof of the theorem is complete.
\end{proof}

%------------------------------------------------

\begin{corollary}\label{G-tilde-equal-S} Let $S$ be defined as in ~\eqref{Definition of the set S}. We have $$\widetilde{G}=S.$$
\end{corollary}

\begin{proof} Clearly $S\subset \widetilde{G}$. Using Theorem~\ref{conjugacy property of sigma}, it follows that $\widetilde{G}\subset S$. Hence the result. 
\end{proof}

In the following remark, we make some pertinent observations about the proof of $\widetilde{G} = S$ (Corollary \ref{G-tilde-equal-S}) comparing it with the case when $\widetilde{G}$ is the $2$-fold metaplectic cover of $G$ considered in \cite{KumAji}.

\begin{remark} \hfill \\
\begin{enumerate}
    \item For $n \geq 3,$ let $\epsilon \in \mu_n$ such that $\epsilon^2 \neq 1.$ Put $h = (g, \epsilon),$ where $g = \begin{bmatrix}
	1 & 1 \\
	0 & 1
\end{bmatrix}$. We will show that $\sigma(h)$ and $h$ are not conjugate in $\widetilde{G}.$ Since $\tau(g) = g,$ we get $\sigma(h) = (g, \epsilon^{-1})$ by Lemma \ref{Properties of sigma} and Remark \ref{sigma-h-computation}. Let $z = (x, \eta) \in \widetilde{G}.$ Then,  $$z h z^{-1} = (xg x^{-1}, \lambda \epsilon),$$ where $\lambda = c(xg, x^{-1}) c(x, g) c(x, x^{-1})^{-1}$.  For $\sigma(h) = z h z^{-1},$ we must have $g = x g x^{-1},$ i.e., $x \in C_{G}(g)$. Note that $ C_{G}(g) = \left\{ \begin{bmatrix}
	a & b \\
	0 & a
\end{bmatrix}: a \in F^{\times}, b \in F \right\}.$ For any $x \in C_{G}(g),$ we have $c(xg, x^{-1}) = \langle a, a \rangle =c(x, x^{-1})$ and $ c(x, g) = 1$ yielding $\lambda =1.$ Thus, $$\sigma(h) \neq z h z^{-1},$$ for any $z \in \widetilde{G}$. Hence, the set $S$ considered in \cite{KumAji} does not work for $n \geq 3,$ i.e.,  $\widetilde{G} \neq S$. \\

\item It is clear that the set $S$ defined in \eqref{Definition of the set S} and considered  in \cite{KumAji} agree when $n=2$. Also, conjugacy invariance of $S$ was used to show that $\widetilde{G} = S$ in the case when $n=2$ (see \cite[Lemma 4.10 \& Theorem 4.11]{KumAji}). \\

\item In the current context, it is not immediate to see whether $S$ is conjugation invariant or not. However, once Corollary \ref{G-tilde-equal-S} is established then it is clear that $S$ is conjugation invariant. \\

\end{enumerate}

\end{remark}

\subsection{Other lifts of the standard involution and their properties} In this section, we define other lifts of the standard involution and show that they also satisfy analogous properties like the lift $\sigma$ of $\tau$. 

\begin{lemma}\label{properties of sigms-alpha} For each $\alpha \in F^{\times}$ and $h\in \widetilde{G}$, let \begin{equation}\label{definiton of sigma-alpha}\sigma_{\alpha}(h)=\langle \alpha, \Delta(h)\rangle \sigma(h).\end{equation} Then 
\begin{enumerate}
\item[\emph{(1)}] $\sigma_{\alpha}$ is an anti-automorphism.
\vspace{0.1 cm}
\item[\emph{(2)}] $\sigma_{\alpha}$ is an involution. 
\vspace{0.1 cm}
\item[\emph{(3)}] $\sigma_{\alpha}$ is a lift of $\tau$. 
\end{enumerate}
\end{lemma}

\begin{proof} \hfill\\

\begin{enumerate}
\item[(1)] For $g,h \in \widetilde{G}$, we have
\begin{align*}
\sigma_{\alpha}(gh) &= \langle \alpha, \Delta(gh)\rangle \sigma(gh)\\
&= \langle \alpha, \Delta(g)\Delta(h)\rangle \sigma(h)\sigma(g)\\
&= \langle \alpha, \Delta(g)\rangle \langle \alpha, \Delta(h)\rangle \sigma(h)\sigma(g)\\
&= \langle \alpha, \Delta(h) \rangle \sigma(h) \langle \alpha, \Delta(g) \rangle \sigma(g)\\
&= \sigma_{\alpha}(h)\sigma_{\alpha}(g). \\
\end{align*}

\item[(2)] Let $\sigma_{\alpha}(h)=y$. We have
\begin{align*}
\sigma_{\alpha}(\sigma_{\alpha}(h)) &= \sigma_{\alpha}(y)\\
&= \langle \alpha, \Delta(y)\rangle \sigma(y)\\
&= \langle \alpha, \Delta(h) \rangle \sigma(\langle \alpha, \Delta(h) \rangle\sigma(h))\\
&= \langle \alpha, \Delta(h) \rangle \langle \alpha, \Delta(h) \rangle ^{-1}\sigma(\sigma(h))\\
%&= \langle \alpha, \Delta(h) \rangle \sigma(\sigma(h))\sigma(\epsilon)\\
%&= \langle \alpha, \Delta(h) \rangle \langle \alpha, \Delta(h) \rangle ^{-1}h\\
&= h. \\
\end{align*}

\item[(3)] Let $h=(g, \epsilon)\in \widetilde{G}$. Suppose that $\sigma((g,1))= (\tau(g), \xi)$. We have
    \[\tau(p(h))= \tau(p((g,\epsilon)))= \tau(g)\]
and 
    \[ p(\sigma(h))= p(\sigma((g,\epsilon)))= p((\tau(g), \epsilon^{-1} \xi))= \tau(g).\]
Thus $\sigma_{\alpha}$ is a lift of $\tau$. 
\end{enumerate}
\end{proof}

\begin{remark} We have used the following observation in part (2) of Lemma~\ref{properties of sigms-alpha}. Let $h=(x, \epsilon)$ and $y=\sigma_{\alpha}(h)$. Then we get $$y=( \tau(x),* ).$$ 
Using the extension of $\Delta$ to $\widetilde{G}$ as defined earlier, we see that $$\Delta(y)=\Delta(\tau(x))=\Delta(h).$$
\end{remark}

It follows from Proposition~\ref{description of the set of liftings} that any lift of $\tau$ is of the form $\sigma_{\alpha}$ for $\alpha\in F^{\times}$. We now show that all the lifts $\sigma_{\alpha}$ also satisfy a similar conjugacy property like $\sigma$. To be precise, we have

\begin{theorem}\label{sigma-alpha-has-similar-propertyas-sigma}
Let $h=(g, \eta) \in \widetilde{G}$. We have \[\sigma_{\alpha}(h)= c(g^{-1},g)^{-2} \eta^{-2} zhz^{-1},\] 
for some $z \in \widetilde{G}$.
\end{theorem}

\begin{proof} Suppose $\Delta(h)\in (F^{\times})^{n}$. Then $\langle \alpha, \Delta(h)\rangle =1$ and $\sigma_{\alpha}(h)=\sigma(h)$. Thus, the theorem holds in this case. Suppose $\Delta(h)\not\in (F^{\times})^{n}$. Taking $\lambda=\alpha^{-1}$ in (i) of Proposition~\ref{properties to check sigma is an involution}, and $u=\tilde{z}(\alpha^{-1})$ we get
\[\langle \alpha, \Delta(h)\rangle ^{-1} h = uhu^{-1}.\]
Indeed, 
\[
    hu = h\tilde{z}(\alpha^{-1}) = \langle \alpha^{-1}, \Delta(h)\rangle ^{-1} \tilde{z}(\alpha^{-1})h = \langle \alpha, \Delta(h)\rangle u h
.\]
Thus,
\begin{align*}\sigma_{\alpha}(h) &= \langle \alpha, \Delta(h)\rangle \sigma(h) \\ 
&= \sigma(\langle \alpha, \Delta(h)\rangle^{-1}h) \\
&= \sigma(uhu^{-1}) \\
&= \sigma(u)^{-1}\sigma(h) \sigma(u).
\end{align*}
Since $\widetilde{G}=S$, we have $\sigma(h)= c(g^{-1},g)^{-2} \eta^{-2}vhv^{-1}$, for some $v \in \widetilde{G}$. \\\\
Thus, we get
\[
    \sigma_{\alpha}(h) =c(g^{-1},g)^{-2} \eta^{-2} \sigma(u)^{-1} v h v^{-1} \sigma(u).
\]
Now taking $z=\sigma(u)^{-1} v$, we get
\begin{align*}
    \sigma_{\alpha}(h)=c(g^{-1},g)^{-2} \eta^{-2} z h z^{-1}.
\end{align*}
Hence the result. 
\end{proof}

\section{Main Theorem}\label{main theorem}

Throughout, we let $\mu_{\widetilde{G}}$ denote the Haar measure on $\widetilde{G}$. We write $\Aut_{c}(\widetilde{G})$ for the group of continuous automorphisms of $\widetilde{G}$ and $\mathbb{R}^{\times}_{> 0}$ for the multiplicative group of positive real numbers. \\

\begin{lemma}\label{results-on-haarmeasure}
  Let $\gamma\in \Aut_{c}(\widetilde{G})$. Then, the following statements hold:
  \begin{enumerate}
      \item[\emph{(1)}] There exists $c_{\gamma}>0$ such that $\mu_{\widetilde{G}}\circ \gamma = c_{\gamma}\mu_{\widetilde{G}}.$ 
      \smallskip

      \item[\emph{(2)}] The map $\gamma\mapsto c_{\gamma}$ from $\Aut_{c}(\widetilde{G})$ to $\mathbb{R}^{\times}_{>0}$ is a homomorphism.
      \smallskip

      \item[\emph{(3)}]   If $\gamma^{2}=1,$ then  $\mu_{\widetilde{G}}\circ \gamma =\mu_{\widetilde{G}}.$
  \end{enumerate}
\end{lemma}

\begin{proof} We refer the reader to Section 5 (Lemma 5.7, Lemma 5.8 and Lemma 5.9) in \cite{KumAji} for the proof. 
\end{proof}

\begin{lemma}\label{restriction-mu-n-of-central-character-injective}
	Let $(\pi, V)$ be an irreducible smooth genuine representation of $\widetilde{G}.$ Let $\omega_\pi$ be the central character of $\pi.$ Then, the following statements hold:
	\begin{enumerate}
		\item[\emph{(1)}] The restriction ${\omega_\pi}\big |\mu_n$ of $\omega_\pi$ to $\mu_n$ is injective.
		\smallskip
		
		\item[\emph{(2)}] $\omega_\pi(\eta^2) = 1,$ for every $\eta\in \mu_n$ if and only if $n=2.$
	\end{enumerate}
\end{lemma}

\begin{proof}
	Let $\eta\in \mu_n$. Since $\pi$ is a genuine representation, we have $\pi(\eta) = \xi(\eta) 1_V.$
 Also, we have that $\mu_n$ is contained in the center of $\widetilde{G}$. Thus applying Schur's lemma, we get $\pi(\eta) = \omega_\pi(\eta)1_V$. It follows that $ \omega_\pi(\eta) = \xi(\eta).$ This proves (1). If $n=2$, clearly we have $\omega_\pi(\eta^2) = 1.$ Suppose that $\omega_\pi(\eta^2) = 1.$ This implies that  $\omega_\pi(\eta) = \pm 1.$  If $n \geq 3,$ there exist $\eta_1, \eta_2$ in $\mu_n$ with $\eta_1 \neq \eta_2$ such that $\omega_\pi(\eta_1) = \omega_\pi(\eta_2),$ which contradicts part (1). This completes the proof of (2).
\end{proof}

\begin{lemma}\label{character-contragredient}
For $g\in \widetilde{G}_{\reg}$, we have $$\Theta_{\pi^{\vee}}(g)=\Theta_{\pi}(g^{-1}).$$
\end{lemma}

\begin{proof} 
We refer the reader to Section 5 (Lemma 5.11) in \cite{KumAji} for the proof. 
\end{proof}

\begin{lemma}
 Let $\rho \in \Aut_{c}(\widetilde{G})$ such that it is a lift of the automorphism  $h \mapsto \tau(h^{-1})$ of $G.$  Then,  
 $$\rho(g) \text{ and } \eta^2 g^{-1} \text{are in }  \widetilde{G}_{\reg},$$
 for any $g\in \widetilde{G}_{\reg}$ and  $\eta \in \mu_n.$
\end{lemma}

\begin{proof}
  Let $\eta \in \mu_n$ and  $g = (x, \epsilon) \in \widetilde{G}_{\reg}.$   Since $x \in G_{\reg}$ and $p(\eta^2 g^{-1}) = x^{-1}$, it follows that $\eta^2 g^{-1} \in \widetilde{G}_{\reg}$. Since $\rho$ is a lift of the automorphism $h \mapsto \tau(h^{-1})$, clearly we have $p(\rho(g))=\tau(x^{-1}) \in G_{\reg}$. Thus $\rho(g) \in \widetilde{G}_{\reg}.$
\end{proof}

Let $\pi$ be an irreducible admissible genuine representation of $\widetilde{G}$. For $f\in C_{c}^{\infty}(\widetilde{G})$, $\rho\in \Aut(\widetilde{G})$ we define 
\[f^{\rho}(g)=f(\rho(g)) \text{ and } \pi^{\rho}(g)=\pi(\rho(g)).\]

\begin{proposition}\label{rho-dualizing-involution}
   Let $\rho \in \Aut_{c}(\widetilde{G})$ such that $\rho^2 =1$ and $\rho(g)$ is conjugate to $\eta^2g^{-1},$ for any $g = (x, \eta)\in \widetilde{G}$. Also, assume that $\rho$ is a lift of the automorphism  $h \mapsto \tau(h^{-1})$ of $G.$  Let $\pi$ be  an irreducible admissible genuine representation of $\widetilde{G}$ and $\omega_\pi$ be the central character of $\pi.$
   Suppose that $\pi^{\rho}$ is isomorphic to $\pi^{\vee}.$ Then 
   \[\omega_{\pi}(\eta^2) =1,\] 
   for all $\eta \in \mu_n.$  
\end{proposition} 

\begin{proof}

 Let   $f\in C_{c}^{\infty}(\widetilde{G})$. Since $\pi^{\rho}\cong \pi^{\vee},$ it follows that $\Theta_{\pi^{\rho}}(f) = \Theta_{\pi^{\vee}}(f).$  It is straightforward to verify that $\Theta_{\pi^{\rho}}(f) = \Theta_{\pi}(f^{\rho}).$  We have, 
 \begin{align}\label{pi-rho-isomorphic-pi-check-1}
      \Theta_{\pi}(f^{\rho}) &=  \int_{\widetilde{G}} f^{\rho}(g) \Theta_{\pi}(g)dg \nonumber \\
      &= \int_{\widetilde{G}} f(\rho(g)) \Theta_{\pi}(g)dg \nonumber \\
       &=  \int_{\widetilde{G}} f(g) \Theta_{\pi}(\rho(g))dg ~~~~~\mbox{($g \mapsto \rho(g)$ and using (3) of Lemma \ref{results-on-haarmeasure})}.
             \end{align}

       Let $g = (x, \eta) \in \widetilde{G}_{\reg}.$ Since $\rho(g)$ is conjugate to $\eta^2 g^{-1}$ and $\Theta_{\pi}$ is constant on (regular semi-simple) conjugacy classes, \eqref{pi-rho-isomorphic-pi-check-1} yields \begin{align}\label{pi-rho-isomorphic-pi-check-3}
      \Theta_{\pi}(f^{\rho}) &=  \int_{\widetilde{G}} f(g) \Theta_{\pi}(\eta^2 g^{-1})dg.
       \end{align}
Also, \begin{align}\label{pi-rho-isomorphic-pi-check-4}
      \Theta_{\pi^{\vee}}(f) &=  \int_{\widetilde{G}} f(g) \Theta_{\pi^{\vee}}( g)dg \nonumber \\
       &= \int_{\widetilde{G}} f(g) \Theta_{\pi}( g^{-1})dg ~~~~~~\mbox{ (by Lemma \ref{character-contragredient})}.
       \end{align}
It follows from \eqref{pi-rho-isomorphic-pi-check-3} and \eqref{pi-rho-isomorphic-pi-check-4} that 
\begin{equation}\label{theta-pi-equalon-etasquareg-inverse-and-ginverse}
    \Theta_{\pi}(\eta^2 g^{-1}) = \Theta_{\pi}(g^{-1}).
\end{equation}
      
       By Lemma \ref{character-non-zero}, $\Theta_{\pi}((g_0, 1)) \neq 0$ for some $g_0 \in G_{\reg}.$ For $ \eta \in \mu_n,$  we have $(g_0, \eta) \in \widetilde{G}_{\reg}.$ By \eqref{theta-pi-equalon-etasquareg-inverse-and-ginverse}, $\Theta_{\pi}(\eta^2 (g_0^{-1}, \eta)^{-1}) = \Theta_{\pi}((g_0^{-1}, \eta)^{-1}).$ 
      Using (2) of Lemma \ref{central-action} proof follows.
  
\end{proof}

\subsection{Proof of the Main Theorem}\label{Proof-of-Main-Theorem}

 Let $g = (x, \eta) \in \widetilde{G}.$ For $\alpha\in F^{\times},$  let $\rho_{\alpha}(g)= \sigma_{\alpha}(g^{-1})$, where $\sigma_{\alpha}$ is given as in \eqref{definiton of sigma-alpha}. Clearly, $\rho_{\alpha}^2=1.$ Since $\sigma_{\alpha}$ is continuous (see Proposition~\ref{continuity of lift}), it follows that $\rho_{\alpha}\in \Aut_{c}(\widetilde{G})$ for all $\alpha\in F^{\times}$. Using Theorem \ref{sigma-alpha-has-similar-propertyas-sigma}, it follows that there exists $z\in \widetilde{G}$ such that $\rho_{\alpha}(g)$ is a conjugate of $\eta^2 g^{-1}$. Indeed, we have 
\begin{align*}
\rho_{\alpha}(g) &= \sigma_{\alpha}((x^{-1}, c(x, x^{-1})^{-1} \eta^{-1}))\\
&= c(x, x^{-1})^{-2} (c(x, x^{-1})^{-1} \eta^{-1})^{-2} z  g^{-1} z^{-1} \\
&= z (\eta^2 g^{-1}) z^{-1}. 
\end{align*}
By Proposition~\ref{rho-dualizing-involution} and part (2) of Lemma \ref{restriction-mu-n-of-central-character-injective}, we have $n=2$. Conversely, if $n=2$ then for each $\alpha\in F^{\times}$, $\sigma_{\alpha}$ is a dualizing involution of $\widetilde{G}$ (see \cite{KumAji} and \cite{KumEkt}).

\bibliographystyle{amsplain}
\bibliography{metaplectic}
\end{document}